\documentclass[12pt,leqno]{article}

\usepackage[cp1250]{inputenc}
\usepackage[OT4]{fontenc}
\usepackage{amsfonts}
\usepackage{amsmath}
\usepackage{amssymb}
\usepackage{amsthm}
\usepackage{fancyhdr}
\usepackage[misc]{ifsym}
\newcommand{\R}{\mathbb{R}}
\newcommand{\Z}{\mathbb{Z}}

\newcommand{\Opis}{{\cal O}}

\newcommand{\Apis}{{\cal A}}

\newcommand{\codim}{\mathop{\rm codim}\nolimits}
\newcommand{\rank}{\mathop{\rm rank}\nolimits}

\newcommand{\sgn}{\mathop{\rm sgn}\nolimits}

\newcommand{\inv}{^{-1}}

\newtheorem{theorem}{Theorem}[section]
\newtheorem{lemma}[theorem]{Lemma}
\newtheorem{cor}[theorem]{Corollary}
\newtheorem{prop}[theorem] {Proposition}

\theoremstyle{definition}
\newtheorem{definition}[theorem]{Definition}
\newtheorem{rem}[theorem]{Remark}
\newtheorem{ex}[theorem]{Example}

\title{\thanks{%
Iwona Krzy\.{z}anowska $({\textrm{\Letter}})$ and Aleksandra~Nowel\\
University of Gda\'{n}sk,
              Institute of Mathematics\\ 
              Faculty of Mathematics, Physics and Informatics\\ 
              University of Gda\'{n}sk\\ 
              Wita Stwosza 57\\ 
              80-308 Gda\'{n}sk\\
              Poland\\
              Email: Iwona.Krzyzanowska@mat.ug.edu.pl $({\textrm{\Letter}})$\\
              Email: Aleksandra.Nowel@mat.ug.edu.pl\\ \\
{\em Keywords:} intersection number, corank 1 matrices, homotopy invariant, cross--cap singularity, polynomial mapping\\              
2010 \emph{Mathematics Subject Classification} 14P25 57R45 12Y05}
Intersection number of a map with the set of matrices of positive corank
}

\author{Iwona~Krzy\.{z}anowska \and Aleksandra~Nowel}

\date{2020}

\begin{document}

\def\nothanksmarks{\def\thanks##1{\protect\footnotetext[0]{\kern-\bibindent##1}}}

\nothanksmarks

\maketitle

\pagestyle{fancy}

\lhead{\fancyplain{}{\textsc{\small Krzy\.{z}anowska, Nowel}}}
\rhead{\fancyplain{}{\emph{\small Intersection number}}}

\begin{abstract}
The definition of the intersection number of a map with a closed manifold can be extended  to the case of a closed stratified set such that the difference between dimensions of its two biggest strata is greater than $1$. The set $\Sigma$ of matrices of positive  corank is an example of such a set. It turns out that the intersection number of a map from an $(n-k+1)$--dimensional manifold with boundary into the set of $n\times k$ real matrices with $\Sigma$ coincides with a homotopy invariant associated with a map going to the Stiefel manifold $\widetilde{V}_k(\R^n)$. In a polynomial case, we present an effective method to compute this intersection number. We also show how to use it to count the number mod $2$ or the algebraic sum of cross--cap singularities of a map from an $m$--dimensional manifold with boundary to $\R^{2m-1}$. 
\end{abstract}

\section{Introduction}

One of the most important tools in topology is the topological degree of a map $f\colon M \to N$ between compact $m$--manifolds. Due to \cite{lecsza1, lecsza2, szafraniec1} we have effective algorithms to compute the topological degree in the polynomial case. 
In fact, the topological degree can be treated as a special case of a more general geometrical concept, called the intersection number (see \cite{hirsch}). Hirsch (\cite{hirsch}) presented the intersection number of a map $M \to W$ with a closed $n$--dimensional submanifold $N\subset W$, where $M$ and $W$ are manifolds, $M$ is compact, $\dim M=m$, $\dim W=n+m$. It is a homotopy invariant.

However, there are situations, where this invariant could be useful, but it cannot be used, because a mapping intersects a set which is not a closed manifold. A good example of such a situation is a map to the set of matrices intersecting the set of matrices of positive corank. In this paper we extend the definition of intersection number to the case where $N$ is a closed stratified subset of $W$ such that the difference between dimensions of the two biggest strata is greater than $1$ (see Section \ref{indeks}). We denote this intersection number of $f\colon M \to W$ with $N$ by $\operatorname{I}(f,N)$ or $\operatorname{I}_2(f,N)$, depending on the orientability of manifolds. Again we obtain a homotopy invariant. 

In case of a mapping $a$ from an $(n-k+1)$--dimensional compact manifold $M$ with boundary to $M_k(\R^n)$ (the set of $n\times k$ real matrices), the intersection number with the set $\Sigma$ of matrices of corank $\geqslant 1$ is well defined. In Section \ref{inter_matrix} we present a nice characterization of this intersection number. We show that for $a|\partial M$ it coincides with invariants presented in  \cite{krzyznowel} (when $M$ is a closed ball, $n-k$ odd) and in \cite{krzyzszafran} ($n-k$ even), where the authors defined a homotopy invariant $\Lambda$ associated with a map from an $(n-k)$--dimensional boundaryless manifold into the Stiefel manifold $\widetilde{V}_k(\R^n)$. When the domain is a sphere, $\Lambda$ induces an isomorphism between $\Z_2$ or $\Z$ (depending on the parity of $n-k$) and $(n-k)$--th homotopy group of $\widetilde{V}_k(\R^n)$.  

In Section \ref{counting} we present effective methods to compute the intersection number modulo $2$ of a polynomial map $a\colon M\to M_k(\R^n)$. We express $\operatorname{I}_2(a, \Sigma)$ in terms of signs of determinants of matrices of some quadratic forms (see Theorem \ref{formula}). So we obtain an easy way to verify whether two such maps are not homotopic.

The main applications of the intersection number defined in this article are showed in Section \ref{applications}. 
In \cite{whitney2, whitney1} Whitney studied general singularities (cross--caps) of mappings from an $m$--dimensional manifold to $\R^{2m-1}$. When $m$ is odd, he associated a sign with a cross--cap singularity.  In \cite{whitney1}, Whitney proved that if $M$ is closed and $f$ has only cross--caps as singularities then the number of cross--caps is even, moreover if $m$ is odd, then the algebraic sum (i.e. the sum of signs) of them equals zero. 
If  $M$ has a boundary then following \cite[Theorem 4]{whitney1}, for a homotopy $f_t\colon
M\longrightarrow \R^{2m-1}$ regular in some open neighbourhood of $\partial M$, if
the only singular points of $f_0$ and $f_1$ are cross--caps then the numbers of
cross--caps of $f_0$ and $f_1$ are congruent mod $2$, and if $m$ is odd, then the algebraic sums of cross--caps of $f_0$ and $f_1$ are the same. 
We present methods how to  effectively check if $f$ has only cross--caps as singular points. Then in the polynomial case using \cite{becker, pedersenetal} one can compute the number of them. We also express the number of cross--caps modulo $2$, and for $m$ odd the algebraic sum of them, in terms of the intersection number presented in this paper. So in the polynomial case, using \cite{krzyzszafran}, and resp. the results presented in Section \ref{counting} one can compute them using for example \textsc{Singular}. In the map--germ case criteria for singularities, including cross--caps, are given in \cite{Saji}.

\section{Preliminaries}
\label{prel}

In this article manifold means a smooth manifold. By $id$ we mean an identity map, \# denotes the number of elements. For the mapping $f$ depending on variables $(x,y)\in \R^m\times \R^l$ we denote by $\frac{\partial f}{\partial x}$ the matrix of partial derivatives of $f$ with respect to the variables $x$.

We will use some properties of the topological degree $\deg g$ (resp. topological degree modulo $2$ --- $\deg _2 g$) of a map $g$ between two compact manifolds. 

Let us recall that if $M$ is an oriented $m$--manifold, $f\colon M\longrightarrow
\R^m$, and $p\in M\setminus \partial M$ is isolated in $f\inv (0)$, then there exists a compact
$m$--manifold $K\subset M\setminus \partial M$ with boundary such that $f\inv (0)\cap K=\{ p\}$ and
$f\inv (0)\cap \partial K=\emptyset$. By $\deg (f,p)$ we will denote the local
topological degree of $f$ at $p$, i.e. the topological degree of the mapping
$\partial K\ni x \mapsto f(x)/|f(x)|\in S^{m-1}$ (for properties of $\deg (f,p)$ see
for example \cite{nirenberg}).

If $g\colon M\longrightarrow \R^m$ is close enough to $f$, then $g\inv (0)\cap
\partial K$ is also empty, and the topological degree of the mapping $\partial K\ni
x \mapsto g(x)/|g(x)|\in S^{m-1}$ is equal to $\deg (f,p)$.

If $M$ is not oriented, then  the local topological degree modulo $2$ --- $\deg_2(f,p)$ of $f$ at $p$ is defined similarly using $\deg_2f$.

If $f\colon (M,\partial M)\to (\R^m,\R^m\setminus \{0\})$, then by $\deg (f|\partial M)$ (resp. $\deg _2(f|\partial M)$) we denote the topological degree (resp. the topological degree modulo $2$) of the map $\partial M\ni x \mapsto f(x)/|f(x)|\in S^{m-1}$. 
 
\bigskip

In this article we will investigate mappings into $M_k(\R^n)$.

\medskip

Take $n-k>0$. 
Let $M_k(\R^n)$ be the space of $n\times k$ real matrices. We identify it with $\R^{kn}$ writing down elements by columns. By $\Sigma^i\subset M_k(\R^n)$ we denote the set of matrices of corank $i$. Put $\Sigma=\bigcup_{i=1}^k\Sigma^i$. 

Note that $\Sigma$ is closed and each $\Sigma^i$ is a submanifold of $M_k(\R^n)$ of codimension $(n-k+i)i$ (see \cite[Chapter II, Proposition 5.3]{golub}). Moreover $\Sigma^i$'s are orientable for $n-k$ even, and non--orientable for $n-k$ odd (see \cite{ando}). Then the stratification $\Sigma=\bigcup_{i=1}^k\Sigma^i$ has properties described at the beginning of Section \ref{indeks}, where $N_{nk-(n-k+i)i}=\Sigma^i$, and other $N_j=\emptyset$.

Let now $M$ be an $(n-k+1)$--dimensional compact manifold with boundary. 
Take a smooth mapping $a:M\longrightarrow M_k(\R^n)$. By $a_i$ we denote its $i$-th column, and by $a_{i}^j(x)$ --- the elements of the matrix $a(x)$ (standing in the $j$--th row and $i$--th column). From now on we will assume that 
\begin{equation}
a(\partial M)\cap \Sigma=\emptyset . \label{zalozenia}
\end{equation}

With the mapping $a$ we may associate (as in \cite{krzyznowel, krzyzszafran}) the mapping $\widetilde{a}:S^{k-1}\times M\longrightarrow\R^n$ given by 
$$\widetilde{a}(\beta,x)=\beta_1a_1(x)+\ldots+\beta_ka_k(x),$$ 
where $\beta=(\beta_1,\ldots ,\beta_k)\in S^{k-1}$. There $S^{k-1}$ denotes the $(k-1)$--dimensional sphere with radius $1$ centered at the origin. Note that $\widetilde{a}(-\beta, x)=-\widetilde{a}(\beta,x)$, and $a(x)\in \Sigma$ if and only if there is $\beta\in S^{k-1}$ such that $\widetilde{a}(\beta,x)=0$. Moreover with each $x\in a\inv (\Sigma^1)$ correspond exactly two elements $(\beta,x), (-\beta,x)\in \widetilde{a}\inv (0)$. 

According to \cite[Theorem 2.3]{krzyznowel}, $a\pitchfork \Sigma^1$ if and only if $0$ is a regular value of the mapping $\widetilde{a}$.

\section{Intersection number of a map with a closed stratified set}
\label{indeks}

Let $W$ be an $(m+n)$--dimensional manifold and $N=\bigcup_{i=0}^nN_i$ be a closed stratified subset of $W$, where  $N_n$ is an $n$--dimensional submanifold of $W$, $N_{n-1}=\emptyset$ and $N_i$ is either an $i$--dimensional submanifold of $W$ or an empty set, for $i=0,\ldots ,n-2.$ In this Section we will define the intersection number of a mapping from an $m$--dimensional compact manifold to $W$ with the set $N$, which coincide with the definition from \cite[Chapter 5.]{hirsch} in case when $N$ is a submanifold of $W$.

For any manifold $S$ with or without boundary, and a smooth map $f\colon S\longrightarrow W$ we will say that $f$ is transversal to $N$ ($f\pitchfork N$) if $f\pitchfork N_i$ for $i=0,\ldots n$. 
Let us see that the set $\{f\in {\cal C}^{\infty}(S,W)\ |\ f\pitchfork N\}=\bigcap\{f\pitchfork N_i\}$. According to \cite[II.4.9, II.4.12]{golub} each set $\{f\pitchfork N_i\}$ contains a residual subset, and since the set of smooth maps from $S$ to $W$ is a Baire space, $\{f\in {\cal C}^{\infty}(S,W)\ |\  f\pitchfork N\}$ is a dense subset of $C^{\infty}(S,W)$.

Take an $m$--dimensional compact manifold  $M$ without boundary and a smooth map $f:M\longrightarrow W$. Then $f\pitchfork N$ if and only if $f\pitchfork N_n$ and $f(M)\cap N_i=\emptyset$ for $i=0,\ldots , n-1$. In that case $f\inv (N)=f\inv(N_n)$, moreover $\codim f\inv(N_n)=\codim N_n=m$, so $f\inv (N_n)$ is a $0$--dimensional submanifold of a compact manifold and so $f\inv (N)$ is finite.

\bigskip

Let $M$, $W$, and $N_n$ be oriented. Take a smooth $f\colon M \longrightarrow W$ such that $f\pitchfork N$. 

\begin{definition}
For $x\in f\inv (N)$ we define $\operatorname{I}(f,N)_x$ as $+1$ or $-1$ depending on whether the mapping 
\[
\eta \colon T_xM \stackrel{df(x)}{\longrightarrow} T_{f(x)}W \longrightarrow T_{f(x)}W/ T_{f(x)}N_n                                                                                                                                                                                                                                                                              \]
preserves or reverses orientation. 

There the orientation of $T_{f(x)}W/ T_{f(x)}N_n$ is the orientation of the normal bundle of $N_n$ at $f(x)$ chosen in such a way, that the orientation of the direct sum of the tangent space and the normal bundle of $N_n$ at $f(x)$ coincides with the orientation of $T_{f(x)}W$. 
\end{definition}

\begin{definition}
The intersection number of $f$ (with $N$) is given by
\[
\operatorname{I}(f,N)=\sum \limits_{x\in f\inv (N)} \operatorname{I}(f,N)_x. 
\]
\end{definition}

\begin{theorem}
If $M$, $W$, $N_n$ are oriented, $f,g:M\longrightarrow W$ are smooth maps such that $f\pitchfork N$, $g\pitchfork N$ and $f$ and $g$ are homotopic, then $\operatorname{I}(f,N)=\operatorname{I}(g,N)$. 
\end{theorem}
\begin{proof}
Take $H:M\times [0,1]\longrightarrow W$ such that $H(\cdot,0)=f$ and  $H(\cdot,1)=g$. We may assume that $H$ is smooth and $H\pitchfork N$. Since $N_{n-1}=\emptyset$, we have $H\pitchfork N_n$ and $H\inv(N_i)=\emptyset$ for $i=1,\ldots n-1$.

Let us see that $H\inv(N)=H\inv(N_n)$ is a compact $1$--dimensional submanifold of $M\times[0,1]$ and $H|_{M\times\{0\}}\pitchfork N$ and $H|_{M\times\{1\}}\pitchfork N$. According to \cite[p. 60]{gupol} we have $$\partial H\inv (N)=H\inv (N)\cap \partial (M\times [0,1])=f\inv (N)\times\{0\} \cup g\inv (N)\times\{1\}.$$ 

Similarly as in \cite{hirsch}, using methods from \cite[Lemma 5.1.1, Lemma 5.1.2]{hirsch}, one may show that
\[
\operatorname{I}(H|\partial (M\times [0,1]),N)=0. 
\]
Since $\operatorname{I}(H|\partial (M\times [0,1]),N)=\operatorname{I}(f,N)-\operatorname{I}(g,N)$, we get $\operatorname{I}(f,N)=\operatorname{I}(g,N)$.
\end{proof}

Now we can define the intersection number of a continuous map $g:M\longrightarrow W$, as $\operatorname{I}(g,N):=\operatorname{I}(f,N)$, where $f\pitchfork N$ and $f$ is smooth and homotopic to $g$.

\bigskip

If $W$, $N_n$ or $M$  is not oriented, one can define the intersection number of the mapping modulo $2$.

\begin{definition}
For $f\pitchfork N$ we define the intersection number of $f$ (with $N$) modulo $2$ as $$\operatorname{I}_2(f,N)=\#f\inv(N)\mod \ 2.$$
\end{definition}

Since the boundary of a compact $1$--dimensional manifold has an even number of points, we get the following.

\begin{theorem} \label{index}
If $f,g:M\longrightarrow W$ are smooth maps such that $f\pitchfork N$, $g\pitchfork N$ and $f$ and $g$ are homotopic, then $\operatorname{I}_2(f,N)=\operatorname{I}_2(g,N)$.
\end{theorem}

Then the intersection number modulo $2$ of a continuous map $g:M\longrightarrow W$ is $\operatorname{I}_2(g,N):=\operatorname{I}_2(f,N)$, where $f\pitchfork N$ and $f$ is smooth and homotopic to $g$.

\medskip

If $M$ is allowed to have a boundary, $\operatorname{I}(f,N)$ and $\operatorname{I}_2(f,N)$ are defined in the same way, but only for maps fulfilling $f(\partial M)\cap N=\emptyset$. In this case we should take the map $g:M\longrightarrow W$ homotopic to $f$ by a homotopy that takes $\partial M$ into $W\setminus N$ at each stage. Of  course $\operatorname{I}(f,N)$ and $\operatorname{I}_2(f,N)$ are invariants only of this kind of homotopy. 

\begin{rem}
Of course if $M$ is a manifold without boundary, and $W=\R^{n+m}$, then the appropriate intersection number of a map $M\to W$ with any $N$ is equal to $0$.
\end{rem}

\section{Intersection number of a map into $M_k(\R^n)$} \label{inter_matrix}

In this section we show that the intersection number of mappings into $M_k(\R^n)$ can be expressed in terms of the topological degree of $\widetilde{a}$ (see Section \ref{prel}). It will turn out (see Section \ref{counting}) that this provides an effective algorithm of computing it in a polynomial case. Moreover, in some special case we connect our intersection number with other invariant associated with a mapping from an $(n-k)$--dimensional sphere to the Stiefel manifold.

\bigskip

If we take a smooth mapping $a$ from an $(n-k+1)$--dimensional compact manifold $M$ with boundary to $M_k(\R^n)$ satisfying (\ref{zalozenia}) as in Section \ref{prel}, then the intersection number $\operatorname{I}(a,\Sigma)$ or $\operatorname{I}_2(a,\Sigma)$ of the map $a$ (see Section \ref{indeks}) is well defined.

\subsection{The case $n-k$ even}

Let $n-k> 0$ be an even number and let $M$ be an $(n-k+1)$--dimensional oriented compact manifold with boundary. In this case $\Sigma ^i$ are orientable (see \cite{ando}). Take a smooth mapping $a:M\longrightarrow M_k(\R^n)$ satisfying (\ref{zalozenia}) in Section \ref{prel}. Note that $a|\partial M:\partial M\longrightarrow \widetilde{V}_k(\R^n)\subset M_k(\R^n)$.

In \cite{krzyzszafran} the authors defined a homotopy invariant $\Lambda$ associated with a map from an $(n-k)$--dimensional boundaryless manifold into the Stiefel manifold $\widetilde{V}_k(\R^n)$. When this manifold is a sphere, $\Lambda$ induces an isomorphism between $\Z$ and $(n-k)$--th homotopy group of $\widetilde{V}_k(\R^n)$.

In our case this $\Lambda (a|\partial M)$ is well defined, and following \cite[Section 2]{krzyzszafran} we will define 
\[
\Lambda (a|\partial M)=\frac{1}{2} \deg (\widetilde{a}|S^{k-1}\times \partial M), 
\]
where $\widetilde{a}|S^{k-1}\times \partial M\colon S^{k-1}\times \partial M \to \R^n\setminus \{0\}$.

\begin{theorem} \label{index_even}
Let $n-k> 0$ be even. There is an orientation of $\Sigma^1$ such that for each smooth manifold $M$ ($(n-k+1)$--dimensional, oriented, compact, with boundary) and smooth $a:M\longrightarrow M_k(\R^n)$ with $a(\partial M)\cap \Sigma=\emptyset$ we have
\[
\operatorname{I}(a,\Sigma)= \Lambda (a|\partial M).
\]

In the case where $a\inv (\Sigma)=a\inv (\Sigma^1)$ is a finite set, 
\[
\operatorname{I}(a,\Sigma)= \frac{1}{2}\sum\limits_{(\beta,x)\in \widetilde{a}\inv(0)}\deg (\widetilde{a},(\beta,x)).
\]
\end{theorem}
\begin{proof}
At the beginning of the proof we fix the orientation of $\Sigma ^1$. In the second step we prove that for the mapping $a$ transversal to $\Sigma ^1$ the intersection number of $a$ and the local topological degree of $\widetilde{a}$ locally coincide. At the end we have to use a technical lemma (see Lemma \ref{technical_matrix}) saying that each point of the intersection of $a$ with $\Sigma ^1$ can be transformed to the form to which we can apply the second step of the proof.

One may represent any matrix $S\in M_k(\R^n)$ as
\begin{equation}
\left [
\begin{array}{cc}
A_{(n-k+1)\times 1} & B_{(n-k+1)\times (k-1)} \cr
C_{(k-1)\times 1} & D_{(k-1)\times (k-1)} 
\end{array}
\right ].  \label{matrix_form}
\end{equation}
If $\det D\neq 0$, then by \cite[Chapter II, Lemma 5.2, Proposition 5.3]{golub} $S\in \Sigma^1$ if and only if $A-BD\inv C=0$. Moreover the map $f(S)=A-BD\inv C\in \R^{n-k+1}$ is such that $f^{-1}(0)$ is locally a complete intersection and coincides with $\Sigma ^1$. Then in a neighbourhood of matrices such that $\det D>0$ the map $f$ gives a natural  orientation of $\Sigma ^1$ in the following way: vectors $v_{n-k+2},\ldots ,v_{nk}$ are well--oriented in the tangent space $T_S\Sigma^1$ if and only if vectors $\nabla f_1(S),\ldots ,\nabla f_{n-k+1}(S),v_{n-k+2},\ldots ,v_{nk}$ are positively oriented in $\R^{nk}$ (there $\nabla f_i$ denotes the gradient of $i$--th coordinate of $f$). Note that $\Sigma ^1$ is connected, and since $n-k$ is even, it is also orientable. Let $\theta$ be the orientation of $\Sigma^1$ that agrees locally with the orientation defined above. From now on we treat $\Sigma^1$ as an oriented manifold $(\Sigma^1,(-1)^{k-1}\theta)$.

By elementary column and row operations each matrix $S\in \Sigma^1$ can be transformed to the form (\ref{matrix_form}), where $A$, $B$, $C$ are zero matrices, and $\det D>0$.

It is sufficient to show the conclusion of this Theorem for $a\pitchfork \Sigma$. Indeed, every $a$ satisfying $a(\partial M)\cap \Sigma=\emptyset$ is homotopic with some mapping which is transversal to $\Sigma$ (by the density of the set of transversal maps). We can choose this homotopy in such a way, that it takes $\partial M$ into $M_k(\R^n)\setminus \Sigma$ (because $a^{-1}(\Sigma)\subset M\setminus \partial M$ is a compact set). Moreover, the intersection number and the invariant $\Lambda$ are invariants of such a homotopy (or its restriction to the boundary), see Section \ref{indeks} and \cite{krzyzszafran}. 

Let us assume that $a\pitchfork \Sigma$. Then $a\inv (\Sigma)=a\inv (\Sigma^1)$ is a finite set. With each point $x\in a\inv (\Sigma)$  we can associate $(\beta, x),(-\beta, x)\in \widetilde{a}\inv (0)$. According to \cite[Theorem 2.3]{krzyznowel} $(\beta, x)$, $(-\beta, x)$ are regular points. By \cite[Proposition 2.4]{krzyzszafran} $\deg (\widetilde{a},(\beta,x))=\deg (\widetilde{a},(-\beta,x))$, moreover $\Lambda (a|\partial M)=\frac{1}{2}\sum \limits_{(\beta,x)\in \widetilde{a}\inv (0)}\deg (\widetilde{a},(\beta,x))$.

Assume that $x\in a\inv (\Sigma)$ is such that
\[
a(x)=\left [
\begin{array}{c|c}
0 &  \cr
\vdots & \mathbf{0}_{(n-k+1)\times (k-1)} \cr
0 &  \cr
\hline
0 &  \cr
\vdots & D_{(k-1)\times (k-1)} \cr
0 &  \cr
\end{array}
\right ],
\] 
where $\det D> 0$. 

Since $x$ is fixed, we may treat $a$ near $x$ as a mapping $(\R^{n-k+1},0)\to M_k(\R^n)$, using a local coordinate system. 

Let us remind that $\operatorname{I}(a,\Sigma)_x$ depends on $\eta\colon T_x M \longrightarrow T_{a(x)}M_k(\R^n)/T_{a(x)}\Sigma^1$. It is easy to verify that the orientation of $T_{a(x)}M_k(\R^n)/T_{a(x)}\Sigma^1$ is given by $n-k+1$ vectors $((-1)^{k-1},0,\ldots ,0),(0,1,\ldots ,0),\ldots ,(0,\ldots ,0,1,0,\ldots ,0)$. So
\[
\operatorname{I}(a,\Sigma)_x= (-1)^{k-1}\sgn \det 
\left [\frac{\partial(a_{1}^1,\ldots ,a^{n-k+1}_1)}{\partial x}(x)\right ].
\]
Note that $((\pm 1,0,\ldots ,0),x)\in \widetilde{a}\inv (0)$. As in \cite[Lemma 3.2]{szafraniec1} we obtain
\[
\deg (\widetilde{a},((1,0,\ldots ,0),x)) = \deg ((\beta_1^2+\ldots +\beta_k^2-1,\widetilde{a}),((1,0,\ldots ,0),x))=
\]
\[
=\sgn \det
\left [
\begin{array}{ccccccc}
1 & 0 & \dots & 0 & 0 & \ldots & 0\cr
 & & a(x) & & & \frac{\partial(a^{1}_1,\ldots ,a^{n}_1)}{\partial x}(x) & 
\end{array}
\right ].
\]
Since $\det D>0$, we can reformulate the last term in the following way:
\[
\sgn \det
\left [
\begin{array}{ccccccc}
1 & 0 & \dots & 0 & 0 & \ldots & 0\cr
 & & a(x) & & & \frac{\partial(a^{1}_1,\ldots ,a^{n}_1)}{\partial x}(x) & 
\end{array}
\right ]=
\]
\[
=(-1)^{k-1}\sgn \det 
\left [\frac{\partial(a^{1}_1,\ldots ,a^{n-k+1}_1)}{\partial x}(x)\right ]=\operatorname{I}(a,\Sigma)_x.
\]
To prove that $\operatorname{I}(a,\Sigma)=\sum \limits_{x\in a\inv (\Sigma)}\operatorname{I}(a,\Sigma)_x=\frac{1}{2}\sum\limits_{(\beta,x)\in \widetilde{a}\inv(0)}\deg (\widetilde{a},(\beta,x))$ it is sufficient to show, that the composition of elementary column and row operations preserves the equality between $\deg (\widetilde{a},(\beta,x))$ and $\operatorname{I}(a,\Sigma)_x$ for $(\beta,x)\in \widetilde{a}\inv (0)$, as is proved in the next Lemma.
\end{proof}

\begin{lemma} \label{technical_matrix}
Let $\Phi \colon M_k(\R^n) \longrightarrow M_k(\R^n)$ be an elementary column or row operation, and $x\in a\inv (\Sigma^1)$, $(\beta,x)\in \widetilde{a}\inv (0)$. Assume that $\operatorname{I}(a,\Sigma)_x=\deg (\widetilde{a},(\beta,x))$. Then there exists such $\bar{\beta}$ that $\widetilde{\Phi(a)}(\bar{\beta},x)=0$, and 
\[
\operatorname{I}(\Phi(a),\Sigma)_x=\deg (\widetilde{\Phi(a)},(\bar{\beta},x)). 
\]
\end{lemma}
\begin{proof}
It is obvious that $\Phi$ and $\Phi|\Sigma ^1\colon \Sigma^1 \longrightarrow \Sigma ^1$ are diffeomorphisms.
Note that the following diagram commutes.
\[
\begin{array}{ccc}
T_xM & \stackrel{\eta_1}{\longrightarrow} & T_{a(x)}M_k(\R^n)/T_{a(x)}\Sigma^1 \\
\shortparallel &       & \downarrow \\  
T_xM & \stackrel{\eta_2}{\longrightarrow} & T_{\Phi(a)(x)}M_k(\R^n)/T_{\Phi(a)(x)}\Sigma^1
\end{array}
\]
So $\operatorname{I}(a,\Sigma)_x=(\sgn s)\cdot  \operatorname{I}(\Phi(a),\Sigma)_x$, where $s$ depends on the way $\Phi$ acts on $M_k(\R^n)$ and $\Sigma^1$.

One can check, that
\begin{enumerate}
\item if $\Phi$ multiplies one column or row by $c\neq 0$, then   
 \begin{itemize}
 \item $\Phi$ reverses the orientation on $M_k(\R^n)$ if and only if $c<0$ and $n$ is odd,
 \item $\Phi|\Sigma ^1$ reverses the orientation on $\Sigma ^1$ if and only if $c<0$ and $n$ is even,
 \end{itemize}
 and in this case $s=c$;
\item if $\Phi$ interchanges two subsequent columns or rows, then 
 \begin{itemize}
 \item $\Phi$ reverses the orientation on $M_k(\R^n)$ if and only if $n$ is odd,
 \item $\Phi|\Sigma ^1$ reverses the orientation on $\Sigma ^1$ if and only if $n$ is even,
 \end{itemize}
 and in this case $s=-1$;
\item if $\Phi$ replaces the se\-cond column (or row) by the sum of the first and second column (or row), then 
 $\Phi$ and $\Phi|\Sigma ^1$ always preserves the orientation,  and in this case $s=1$.
\end{enumerate}

To verify whether $\Phi|\Sigma ^1$ preserves or reverses the orientation, it is enough to check it at some fixed point of $\Phi|\Sigma ^1$.

Note that there exist two diffeomorphisms $\Psi_1\colon S^{k-1}\longrightarrow S^{k-1}$, $\Psi_2\colon \R^n \longrightarrow \R^n$ such that the following diagram commutes
\[
\begin{array}{ccc}
S^{k-1}\times M & \stackrel{\widetilde{a}}{\longrightarrow} & \R^n \\
{\scriptstyle (\Psi_1, id)}\downarrow &       & \downarrow {\scriptstyle \Psi_2} \\  
S^{k-1}\times M & \stackrel{\widetilde{\Phi(a)}}{\longrightarrow} & \R^n,
\end{array}
\]
and $\widetilde{a}(\beta,x)=0 \ \Leftrightarrow \ \widetilde{\Phi(a)}(\Psi_1(\beta),x)=0$. Note that if $\Phi$ is a column operation, then $\Psi_2=id$, and if $\Phi$ is a row operation, then $\Psi_1=id$.

So $\deg(\widetilde{a},(\beta,x))=(\sgn t)\cdot  \deg(\widetilde{\Phi(a)},(\Psi_1(\beta),x))$, where $t$ depends on the way $\Psi_1$ and $\Psi_2$ acts on $S^{k-1}$ and $\R^n$. One can verify that $\sgn t$ and $\sgn s$ coincide. Since $\operatorname{I}(a,\Sigma)_x=\deg (\widetilde{a},(\beta,x))$, we have $\operatorname{I}(\Phi(a),\Sigma)_x=\deg (\widetilde{\Phi(a)},(\bar{\beta},x))$. 
\end{proof}

\subsection{The case $n-k$ odd}

Let $n-k> 0$ be an odd number and let $M$ be an $(n-k+1)$--dimensional compact manifold with boundary, not necessarily orientable. In this case $\Sigma ^i$ are non--orientable (see \cite{ando}).

\begin{theorem} \label{sumastopni}
Let us assume that $a:M\longrightarrow M_k(\R^n)$ is such that $a(\partial M)\cap \Sigma=\emptyset$. The set $\widetilde{a}\inv(0)$ is finite if and only if $a\inv(\Sigma)=a\inv(\Sigma^1)$ and $a\inv(\Sigma^1)$ is finite. If that is the case, then
$$
\operatorname{I}_2(a,\Sigma)=\sum _{(\beta, x)}\deg_2(\widetilde{a},(\beta, x))\mod 2,
$$
where $(\beta, x)\in \widetilde{a} \inv (0)$ and we choose only one from each pair $(\beta, x), (-\beta, x)\in \widetilde{a} \inv (0)$ ($(\beta, x)$ runs through half of the zeros of $\widetilde{a}$). 
\end{theorem}
 
\begin{proof}
It is sufficient to show the conclusion for  $a\pitchfork \Sigma$. Note that then $0$ is a regular value of $\widetilde{a}$ (see \cite{krzyznowel}).

Take  $a\pitchfork \Sigma$, then $\operatorname{I}_2(a,\Sigma)=\# a\inv(\Sigma^1)=\frac12 \# \widetilde{a}\inv (0) \mod 2$.
Since $0$ is a regular value of $\widetilde{a}$, $\deg_2(\widetilde{a},(\beta, x))=1$ for each  $(\beta, x)\in \widetilde{a} \inv (0)$. Hence $\operatorname{I}_2(a,\Sigma)=\sum _{(\beta, x)}\deg_2(\widetilde{a},(\beta, x))\mod 2$, where $(\beta, x)$  runs through half of the zeros of $\widetilde{a}$.
\end{proof}
Note that $\partial (S^{k-1}\times M)=S^{k-1}\times \partial M$,  $\widetilde{a}|_{S^{k-1}\times \partial M}:S^{k-1}\times \partial M\longrightarrow\R^n\setminus\{0\}$ and the modulo $2$ topological degree of $\widetilde{a}|_{S^{k-1}\times \partial M}$ is well defined, but for $n-k$ odd it is always equal to $0$, see \cite{krzyzszafran}.

\begin{rem}
If $n-k$ is even and $M$ is non--orientable, then $\operatorname{I}_2(a,\Sigma)$ is defined, and Theorem \ref{sumastopni} still holds true. 
\end{rem}

In \cite{krzyznowel} the authors defined a homotopy invariant $\Lambda$ associated with a map from an $(n-k)$--dimensional sphere $S^{n-k}$ into the Stiefel manifold $\widetilde{V}_k(\R^n)$, $n-k$ odd. This $\Lambda$ induces an isomorphism between $(n-k)$--th homotopy group of $\widetilde{V}_k(\R^n)$ and $\Z_2$.

Now let us assume that $M=B^{n-k+1}$, the $(n-k+1)$--dimensional ball. Then $a(S^{n-k})\cap \Sigma=\emptyset$, and $\Lambda (a|S^{n-k})$ is well defined. According to Theorem \ref{sumastopni} and \cite[Section 2.]{krzyznowel}, we get

\begin{cor}
\[
\Lambda (a|S^{n-k})=\operatorname{I}_2(a,\Sigma). 
\]
\end{cor}

\section{Counting the intersection number mod 2 for polynomial mappings} \label{counting}

In this Section we present the method to compute effectively $\operatorname{I}_2(a,\Sigma)$ in a polynomial case.

We showed that $\operatorname{I}(a,\Sigma)=\Lambda(a|\partial M)$, so in a polynomial case, according to \cite[Theorem 3.3.]{krzyzszafran}, with some additional assumptions,  $\operatorname{I}(a,\Sigma)$ can be expressed as the half of a sum of signatures of quadratic forms that can be constructed explicitly. This gives us an effective method to compute $\operatorname{I}(a,\Sigma)$.

\bigskip

Take polynomial mappings $h=(h_1,\ldots ,h_l):\R^{n-k+1+l}\longrightarrow\R^{l}$ and $g:\R^{n-k+1+l}\longrightarrow\R$, $n-k>0$ odd. Let us assume that $h\inv(0)$ is a complete intersection (i.e. $h^{-1}(0)\neq \emptyset$ and $\rank dh(x)=l$ for $x\in h^{-1}(0)$) and $M=h^{-1}(0)\cap \{g\geqslant 0\}$. Take a polynomial mapping $a:M\longrightarrow M_k(\R^n)$, such that $a(\partial M)\cap \Sigma=\emptyset$. As in \cite[Lemma 3.2]{szafraniec1} for an isolated zero $(\beta,x) \in\widetilde{a}\inv(0)$ we have $$\deg(\widetilde{a},(\beta,x))=\deg((h,\widetilde{a}),(\beta,x)),$$ where by $(h,\widetilde{a})$ we mean a mapping from $S^{k-1}\times \R^{n-k+1+l}$ to $\R^l\times \R^{n}$,
and so the same equality holds for modulo $2$ local topological degrees.

Let us take an ideal $J$ in $\R[x]=\R[x_1,\ldots,x_{n-k+1+l}]$ generated by $h_1,\ldots ,h_l$ and all $k\times k$ minors of matrix $a(x)$, and let the ideal $J'$ be generated by $h_1,\ldots ,h_l$ and all $(k-1)\times (k-1)$ minors of $a(x)$. Put $\Apis=\R[x]/J$.

From now on we will assume that $\dim_{\R}\Apis<\infty$, $J+J'=\R[x]$, $J+\langle g\rangle=\R[x]$. From the first assumption we get that the zero set $V(J)$ of $J$ is finite, from the second one --- that $a\inv(\Sigma)=a\inv(\Sigma^1)$, and so $\widetilde{a}\inv(0)$ and $a\inv (\Sigma)$ are finite sets. The last assumption implicates that $a(\partial M)\cap \Sigma=\emptyset$.

Denote by $\Opis_x^n$  the ring of germs at $x\in \R^n$ of analytic functions $\R^{n}\longrightarrow \R$, by  $\Opis_x^{S}$ the ring of germs at $x\in S$ of analytic functions $S\longrightarrow \R$.

The following Proposition contains some results from \cite[Section 3]{krzyzszafran}, adapted to our case. 

\begin{prop}\label{isomorphism}
Take $p\in V(J)$ and $(\beta,p)\in\widetilde{a}\inv(0)$. One can construct a polynomial mapping $F=(F_1,\ldots ,F_n):\R^{k-1}\times \R^{n-k+1+l}\longrightarrow\R^n$ such that there is such $\lambda\in\R^{k-1}$ that $(\lambda,p)$ is an  isolated zero of $(h,F)$ with following properties: 
\begin{itemize}
\item $\deg_2(\widetilde{a},{(\beta,p)})=\deg_2((h,F),{(\lambda,p)})$,
\item $\Opis_p^{n-k+1+l}/J\simeq\Opis_{(\lambda,p)}^{n+l}/\langle h,F \rangle$.
\end{itemize}
\end{prop}
\begin{proof} 
Take $p\in V(J)$ and $(\beta,p)\in\widetilde{a}\inv(0)$. Analogically as in \cite{krzyzszafran}, one can construct a polynomial mapping $F$ such that $\deg_2(\widetilde{a},{(\beta,p)})=\deg_2(F|(\R^{k-1}\times h\inv(0)),{(\lambda,p)})$, for some $\lambda$. Then by applying \cite[Lemma 3.2]{szafraniec1} we get $\deg_2(F|(\R^{k-1}\times h\inv(0)),{(\lambda,p)})=\deg_2((h,F),{(\lambda,p)})$.

Moreover in the neighbourhood of $p$ 
there is $\lambda(x)=(\lambda_2(x),\ldots ,\lambda_k(x))$ such that $F_{i_1}(\lambda(x),x)=\ldots =F_{i_{k-1}}(\lambda(x), x)=0$, $1\leqslant i_1<\ldots <i_{k-1}\leqslant n$. 
As in \cite{krzyzszafran} $\Gamma=\{(\lambda(x),x)|\ x\in h\inv(0)\}$ is an $(n-k+1)$--dimensional manifold, and we put $\Omega(x)=(\Omega_1(x),\ldots , \Omega_n(x))=F(\lambda(x),x)$.

It is easy to see that $$\Opis^{n+l}_{(\lambda,p)}/\langle h, F\rangle\simeq \Opis^{\Gamma}_{(\lambda,p)}/\langle  F\rangle\simeq \Opis^{h\inv(0)}_{p}/\langle \Omega\rangle\simeq \Opis^{n-k+1+l}_{p}/\langle h, \Omega\rangle.$$

From \cite[Lemma 3.7]{krzyzszafran} we get that $J=\langle h, \Omega\rangle$ in $\Opis^{n-k+1+l}_p$, and then $$\Opis^{n+l}_{(\lambda,p)}/\langle h, F\rangle\simeq \Opis^{n-k+1+l}_p/J.$$
\end{proof}

\begin{theorem} \label{formula}
Take polynomial mappings $h=(h_1,\ldots h_l):\R^{n-k+1+l}\longrightarrow\R^{l}$ and $g:\R^{n-k+1+l}\longrightarrow\R$, $n-k>0$ odd, such that $h\inv(0)$ is a complete intersection and $M=h^{-1}(0)\cap \{g\geqslant 0\}$ is compact. For a polynomial $a:M\longrightarrow M_k(\R^n)$, let us assume that $\dim_{\R}\Apis<\infty$, $J+J'=\R[x]$, $J+\langle g\rangle=\R[x]$. Then for any linear functional $\varphi \colon \Apis
\longrightarrow \R$ and  $\Phi$, $\Psi$ -- the bilinear
symmetric forms on $\Apis$ given by $\Phi(f_1,f_2)=\varphi(f_1f_2)$,
$\Psi(f_1,f_2)=\varphi(g f_1f_2)$ such that $\det [\Psi]\neq 0$, we have  $\det [\Phi]\neq 0$
and
$$
\operatorname{I}_2(a,\Sigma)=\dim_{\R} \Apis +1+\frac{1}{2}(\sgn \det [\Phi]+\sgn \det[\Psi])\mod 2,
$$
where $[\Phi]$ denotes the matrix of the form $\Phi$.
\end{theorem}
\begin{proof}
Let $V(J)=\{p_1, \ldots ,p_m\}$. There is an even--dimensional algebra $\mathcal{D} $ such that the natural projection 
\[
 {\mathcal A}\to \oplus_{i=1}^m {\mathcal O}_{p_i}^{n-k+1+l}/J \oplus \mathcal{D} 
\]
is an isomorphism of $\R$--algebras (see e.g. \cite[Section 1.]{szafraniec1}), and so
$\dim_{\R}{\mathcal A}= \sum\limits_{i} \dim_{\R}{\mathcal O}_{p_i}^{n-k+1+l}/J \mod 2$. 

By Proposition \ref{isomorphism} we have an isomorphism
\[
\oplus_{i=1}^m {\mathcal O}_{(\lambda_i,p_i)}^{n+l}/\langle h,F\rangle \oplus \mathcal{D} \simeq \mathcal{A}.
\]

According to \cite[Theorem 2.3.]{szafraniec1} we get

$$\sum_{i:g(p_i)>0}\deg_2((h,F),(\lambda_i,p_i))=$$ $$=\dim_{\R} \left(\oplus_{i=1}^m {\mathcal O}_{(\lambda_i,p_i)}^{n+l}/\langle h,F\rangle \right)+1+\frac{1}{2}(\sgn \det [\Phi]+\sgn \det[\Psi])\mod 2 = $$
$$=\dim_{\R} \mathcal{A}+1+\frac{1}{2}(\sgn \det [\Phi]+\sgn \det[\Psi])\mod 2.$$

Since by Propositions \ref{sumastopni}, \ref{isomorphism} $\operatorname{I}_2 (a,\Sigma)=\sum_{i:g(p_i)>0}\deg_2((h,F),(\lambda_i,p_i))$, we get the conclusion.
\end{proof}

\begin{rem}
If we take $g=1$ then $M=h\inv(0)$ is a manifold without boundary, $\Phi=\Psi$, and we get that  $\operatorname{I}_2(a,\Sigma)=\dim_{\R} \Apis =0\mod 2$.
\end{rem}

\begin{rem}
One can also use Theorem \ref{formula} as a simple way to compute $\operatorname{I}(a,\Sigma)$ (and so the invariant from \cite{krzyzszafran}) modulo $2$ in the case where $n-k$ is even (because $\operatorname{I}(a,\Sigma)\mod 2 =\operatorname{I}_2(a,\Sigma)$). 
\end{rem}

Using \textsc{Singular} (\cite{singular}) and previous Theorem we present the following examples.

\begin{ex}
Let $M$ be a half of a $2$--dimensional sphere ($g(x,y,z)=-z$, $h(x,y,z)=x^2+y^2+z^2-1\colon \R^3\to \R$). Take $a,b\colon M\to M_2(\R^3)$ as
\[a(x)=\left [
\begin{matrix}
10x^2z+4xy+10x & 2x^2y+7x^2+6y \\
6xz^2+4y^2+7y & 3z^3+10z^2+9z  \\
4z^3+2z^2+7x & 7xz^2+7xy+2y
\end{matrix}\right ],
\]
\[b(x)=\left [
\begin{matrix}
8x^2y+8z^2+x & 5xyz+2yz+x \\
9xy^2+6z^2+6x & x^2y+z^2+3y \\
2y^3+5y^2+6y & 2x^3+7xy+6y
\end{matrix}\right ].
\]
Then for $a$ the dimension of the algebra ${\cal A}$ equals $48$, $\operatorname{I}_2(a,\Sigma)=0 \mod 2$, and for $b$ the dimension of the algebra ${\cal A}$ equals $50$, $\operatorname{I}_2(b,\Sigma)=1  \mod 2$. 

So $a$ and $b$ are not homotopic by a homotopy that takes $\partial M$ into $M_2(\R^3)\setminus \Sigma$.
\end{ex}

\begin{ex}
Let $M$ be a half of a $2$--dimensional torus ($g(x,y,z,w)=z$, $h(x,y,z,w)=(1-x^2-w^2, 1-y^2-z^2)\colon \R^4\to \R^2$). Take $a\colon M\to M_2(\R^3)$ as
\[a(x)=\left [
\begin{matrix}
w^3x+7w^2x+8y^2+8z & y^4+7xz^2+6z^2+3x \\
xyz^2+5wxy+10wx+4y & x^2y^2+8wy^2+2wy+9w \\
2x^2y^2+7x^3+z^2+10y & 5w^4+8wxy+6w^2+8y
\end{matrix}\right ].
\]
Then the dimension of the algebra ${\cal A}$ equals $184$, and
$\operatorname{I}_2(a,\Sigma)=1 \mod 2$. 
\end{ex}

\section{Applications. Counting the number of cross--caps modulo $2$ and the algebraic sum of them} \label{applications}

Mappings from an $m$--dimensional manifold $M$ into $\R^{2m-1}$ are natural object of study. In \cite{whitney1,whitney2}, Whitney described typical mappings from $M$ into $\R^{2m-1}$. Those mappings have only isolated critical points, called cross--caps. Whitney showed that the algebraic sum of cross--caps ($m$ odd) or the number of cross--caps modulo $2$ ($m$ even) is an important invariant of such mappings. This invariant was investigated in various situations for example in \cite{ikegamisaeki, krzyz, krzyznowel}.

We will show a new method to check whether a point is a cross--cap singularity. In a polynomial case it gives a method of verifying effectively whether a mapping $f$ has only cross--caps as singular points. Moreover, we show, that the algebraic sum of cross--caps or the number of cross--caps modulo $2$ can be expressed as the intersection number of some mapping associated to $f$, and so can be effectively counted using the results of previous sections. 

\bigskip

At the beginning of this section we will present three technical lemmas and then we switch to the applications.

\begin{lemma} \label{tech1}
Let $U\subset \R^{n-k+1}$ be an open set, $n-k>0$, $a\colon U\to M_k(\R^n)$. Take $x\in U$ such that $a(x)\in \Sigma ^1$. Let us assume that $a(x)$ has the following form: 
\[
a_1(x)=(0,\ldots,0) \mbox{ and }
a_i(x)=(0,\ldots,0,a_i^{n-k+2}(x),\ldots,a_i^{n}(x))
\]
for $i=2,\ldots,k$ (here $a_i^j$ is the element standing in the $j$--th row and $i$--th column).
Then $a \pitchfork \Sigma ^1$ at $x$ if and only if
\[
\rank \left [\frac{\partial (a_1^1,\ldots ,a_1^{n-k+1})}{\partial (x_1,\ldots
,x_{n-k+1})}(x)\right ]=n-k+1. 
\]
\end{lemma}

\begin{proof}
As in the proof of Theorem \ref{index_even} one can show that the tangent space $T_{a(x)}\Sigma ^1$ is spanned by vectors
$v_i=(0,\ldots ,0,\ldots , 1, \ldots ,0)$, where $1$ stands at $(i+n-k+1)$--th place, $i=1,\ldots ,nk-(n-k+1)$.

Let us observe that $a\pitchfork \Sigma ^1$ at $x$ if and only if $\rank
da(x)$ is maximal (i.e. equals $n-k+1$) and $T_{a(x)}\Sigma ^1\cap
da(x)\R^{n-k+1}=\{0\}$. It is equivalent to the condition:
\begin{equation} \label{maxrank}
\rank \left [\frac{\partial (a_1^1,\ldots ,a_1^{n-k+1})}{\partial (x_1,\ldots
,x_{n-k+1})}(x)\right ]=n-k+1.
\end{equation}
\end{proof}

\begin{lemma} \label{tech2}
Let $U\subset \R^{n-k+1}$ be an open set, $n-k>0$, $a\colon U\to M_k(\R^n)$, $b\colon U\to M_k(\R^s)$, $c\colon U\to M_s(\R^{n+s})$. We define $e\colon U\to M_{k+s}(\R^{n+s})$ as 
\[
e(x)= \left [
\begin{array}[c]{c|c}
\vbox{\hbox{b(x)} \hbox{a(x)}} & \vbox{\vfill \hbox{c(x)} \vfill} \\
\end{array}
\right ].
\]
Let us assume that for each $x\in U$ we have $\rank a(x)=\rank \left [
\begin{array}{c}
b(x)\\ 
a(x)
\end{array}
\right ]$ and $s+\rank a(x)= \rank e(x)$.
Then $a\pitchfork \Sigma ^1\subset M_k(\R^n)$ at $x$ if and only if $e\pitchfork \Sigma ^1\subset M_{k+s}(\R^{n+s})$ at $x$.
\end{lemma}
\begin{proof}
Take $x\in U$ such that $a(x)\in \Sigma ^1$. Note that then also $e(x)\in \Sigma ^1$. By elementary column and row operations we may transform mappings $a$ and $e$ to such forms, that the first $n-k+1$ elements of the first column of both of them coincide, and at the point $x$ matrices $a(x)$ and $e(x)$ have the following forms:
\[
a(x)=\left [
\begin{array}{c|c}
0 &  \cr
\vdots & \mathbf{0}_{(n-k+1)\times (k-1)} \cr
0 &  \cr
\hline
0 &  \cr
\vdots & D_{(k-1)\times (k-1)} \cr
0 &  \cr
\end{array}
\right ], \quad
e(x)=\left [
\begin{array}{c|c|c}
0 & & \cr
\vdots & \mathbf{0}_{(n-k+1)\times (k-1)} & \mathbf{0}_{(n-k+1)\times s} \cr
0 & & \cr
\hline
0 & & \cr
\vdots & \mathbf{0}_{s\times (k-1)} & E_{s\times s} \cr
0 & & \cr
\hline
0 & & \cr
\vdots & D_{(k-1)\times (k-1)} & \mathbf{0}_{(k-1)\times s}  \cr
0 & & \cr
\end{array}
\right ],
\] 
where $D$ and $E$ have maximal rank. From the previous Lemma we obtain the conclusion.
\end{proof}

\begin{lemma}\label{tech3}
Let $U\subset \R^{n-k+1}$ be an open set, $n-k>0$, $a\colon U\to M_k(\R^n)$, $b\colon U\to M_k(\R^k)$. We define $ab\colon U\to M_k(\R^{n})$ as $ab(x)=a(x)\cdot b(x)$.
Let us assume that for each $x\in U$ we have $\rank b(x)=k$.
Then $a\pitchfork \Sigma ^1$ at $x$ if and only if $ab\pitchfork \Sigma ^1$ at $x$.

For $\bar{x}$ such that $a\pitchfork \Sigma ^1$ at $\bar{x}$, if $\det b(\bar{x})>0$, then $\operatorname{I}(a,\Sigma)_{\bar{x}}=\operatorname{I}(ab,\Sigma)_{\bar{x}}$, resp. $\operatorname{I}_2(a,\Sigma)_{\bar{x}}=\operatorname{I}_2(ab,\Sigma)_{\bar{x}}$.
\end{lemma}
\begin{proof}
Take $x\in U$ such that $a(x)\in \Sigma ^1$. Note that then also $ab(x)\in \Sigma ^1$. By elementary column and row operations we may transform mappings $a$ and $ab$ to such forms, that matrices $a(x)$ and $ab(x)$ have the following forms:
\[
a(x)=\left [
\begin{array}{c|c}
0 &  \cr
\vdots & \mathbf{0}_{(n-k+1)\times (k-1)} \cr
0 &  \cr
\hline
0 &  \cr
\vdots & D_{(k-1)\times (k-1)} \cr
0 &  \cr
\end{array}
\right ], \quad
ab(x)=\left [
\begin{array}{c|c}
0 &  \cr
\vdots & \mathbf{0}_{(n-k+1)\times (k-1)} \cr
0 &  \cr
\hline
0 &  \cr
\vdots & \bar{D}_{(k-1)\times (k-1)} \cr
0 &  \cr
\end{array}
\right ],
\] 
where $D$ and $\bar{D}$ have maximal rank. 

Note that $(ab)_1^j=\sum \limits_{r=1}^k a_r^jb_1^r$. From the forms of the matrices $a(x)$ and $ab(x)$ we get $b_1^1(x)\neq 0$, $b_1^2(x)=\ldots =b_1^{k}(x)=0$, $a_r^j(x)=0$ for $j=1,\ldots, n-k+1$, $r=1,\ldots,k$. Hence
\[
\left [ \frac{\partial ((ab)_1^1,\ldots ,(ab)_1^{n-k+1})}{\partial (x_1,\ldots ,x_{n-k+1})}(x)\right ]=b_1^1(x)\cdot \left [ \frac{\partial (a_1^1,\ldots ,a_1^{n-k+1})}{\partial (x_1,\ldots ,x_{n-k+1})}(x)\right ]. 
\]
From Lemma \ref{tech1} we obtain the conclusion. 

There is $r>0$ such that $\bar{B}(\bar{x},r)\subset U$ and $a\inv(\Sigma)\cap \bar{B}(\bar{x},r)=\{\bar{x}\}$ (where $\bar{B}(\bar{x},r)$ denotes the closed ball centered at $\bar{x}$ of the radius $r$). Since the set of $k\times k$--matrices with positive determinant is path--connected, then there exists a homotopy $h\colon \bar{B}(\bar{x},r)\times [0,1]\to M_k(\R^k)$ between the constant map equal the identity matrix and $b$. Then $ah\colon \bar{B}(\bar{x},r)\times [0,1]\to M_k(\R^n)$ given by
\[
ah(x,t)=a(x)\cdot h(x,t) 
\]
is a homotopy between $a$ and $ab$ such that $ah\inv (\Sigma)\subset B(\bar{x},r)\times [0,1]$.  
Since $\operatorname{I}$ and $\operatorname{I}_2$ are homotopy invariants, we get the conclusion.
\end{proof}

Let $M$ be a smooth $m$--dimensional manifold.
According to  \cite{golub,whitney2,whitney1}, a point $p\in M$ is a cross--cap of a
smooth mapping $f:M\longrightarrow \R^{2m-1}$, if there is a coordinate system near
$p$, such that
in some neighbourhood of $p$ the mapping $f$ has the form
\[(x_1,\ldots , x_m)\mapsto(x_1^2,x_2,\ldots , x_m,x_1x_2,\ldots ,x_1x_m).\]

Take $f:M\longrightarrow \R^{2m-1}$ with only cross--caps as singularities.
By \cite[Theorem 3]{whitney1}, if $m$ is even and $M$ is a closed manifold, then $f$ has an even
number of cross--caps.

Let $(M,\partial M)$ be an $m$--dimensional smooth compact manifold with boundary, $m$
\textbf{even}. Take a continuous mapping $f\colon [0,1]\times M\longrightarrow \R^{2m-1}$ such that
there exists a neighbourhood of $\partial M$ in which all the $f_t$'s are regular,
and $f_0$, $f_1$ have only cross--caps as singularities. According to  \cite[Theorem
4]{whitney1}, mappings $f_0$ and $f_1$ have the same number of cross--caps $\mod 2$.

Let $(M,\partial M)$ be an $m$--dimensional smooth compact manifold with boundary, $m$
\textbf{odd}. Take a smooth mapping $f:M\longrightarrow\R^{2m-1}$ and  let $p\in M$ be a cross--cap of $f$. 
According to \cite{whitney1}, there are  coordinate systems near $p$ and $f(p)$, such that 
\begin{equation} \label{warunek1}
\frac{\partial f}{\partial x_1}(p)=0
\end{equation} and vectors 
\begin{equation}\label{wektory}
\frac{\partial^2 f}{\partial x_1^2}(p),\frac{\partial f}{\partial x_2}(p),\ldots ,\frac{\partial f}{\partial x_m}(p),\frac{\partial^2 f}{\partial x_1\partial x_2}(p),\ldots , \frac{\partial^2 f}{\partial x_1\partial x_m}(p)
\end{equation} 
are linearly independent.
The cross--cap $p$ is called positive (negative) if  the vectors (\ref{wektory}) determine the negative  (positive) orientation of $\R^{2m-1}$. According to \cite[Lemma 3]{whitney1}, this definition does not depend on choosing the coordinate system on $M$. 

Take a continuous mapping $f\colon [0,1]\times M\longrightarrow \R^{2m-1}$ such that
there exists a neighbourhood of $\partial M$ in which all the $f_t$'s are regular,
and $f_0$, $f_1$ have only cross--caps as singularities. According to  \cite[Theorem
4]{whitney1}, mappings $f_0$ and $f_1$ have the same algebraic sum of cross--caps.

For any smooth map $g\colon \R^n\to\R^s$ we consider its tangent map $dg$ as a map going from $\R^n$ to $M_s(\R^{n})$.

Let $h=(h_1,\ldots ,h_l)\colon \R^{m+l}\to \R^l$ be a smooth map such that $h^{-1}(0)\neq \emptyset$ and $\rank dh(x)=l$ for $x\in h^{-1}(0)$, i.e. $h\inv(0)$ is a complete intersection and so a smooth manifold of dimension $m$.

\begin{prop} \label{cross-cap1}
Let $f=(f_1,\ldots ,f_{2m-1})\colon \R^{m+l}\to \R^{2m-1}$ be a smooth map. Then $p\in h\inv(0)$ is a cross--cap of $f|h\inv(0)$ if and only if $\rank d(h,f)(p)=m+l-1$ and $\left (d(h,f)\right )|h\inv(0) \pitchfork \Sigma^1$ at $p$. 
\end{prop}
\begin{proof}
Note that if $p$ is a cross--cap of $f|h^{-1}(0)$, then $\rank d(f|h^{-1}(0))(p)=m-1$, and this holds if and only if $\rank d(h,f)(p)=m+l-1$. Let $p$ be such that $\rank d(h,f)(p)=m+l-1$.

Let us take a local coordinate system $\varphi\colon (\R^{m+l},0) \to (\R^{m+l},p)$ such that $(h\inv(0),p)=\varphi((\R^m\times \{0\},0))$. Note that $h\circ \varphi|(\R^m\times \{0\},0)\equiv 0$.

The point $p$ is a cross--cap of $f|h\inv(0)$ if and only if $0$ is a cross--cap of $f\circ \varphi|(\R^m\times \{0\},0)$. By \cite[Remark, p. 14]{krzyznowel} this takes place if and only if $d\left (f\circ \varphi | (\R^m\times \{0\},0)\right ) \pitchfork \Sigma ^1$ at $0$. From Lemma \ref{tech2} it holds if and only if $\left (d ((h,f)\circ\varphi)\right )| (\R^m\times \{0\},0)\pitchfork \Sigma ^1$ at $0$. 

Let us observe that $dh(\varphi(x))\cdot \left [ \frac{\partial \varphi}{\partial y}(x)\right ]$ is a square matrix with maximal rank, and
\[
d\left ((h,f)\circ\varphi\right )(0)=d(h,f)(p)\cdot d\varphi(0)=
\left [ 
\begin{matrix}
\mathbf{0} & dh(p)\cdot \frac{\partial \varphi}{\partial y}(0)\cr
df(p)\cdot \frac{\partial \varphi}{\partial x}(0) & *
\end{matrix}
\right ],
\]
where $(x,y)$ are coordinates in $\R^m\times \R^l=\R^{m+l}$.

Since $d\left ((h,f)\circ\varphi\right )(x,0)=d(h,f)(\varphi(x,0))\cdot d\varphi(x,0)$, from Lemma \ref{tech3} we obtain that $\left (d ((h,f)\circ\varphi)\right )| (\R^m\times \{0\},0)\pitchfork \Sigma ^1$ at $0$ if and only if $\left (d(h,f)\right )\circ \varphi | (\R^m\times \{0\},0)\pitchfork \Sigma ^1$ at $0$, which is equivalent to $\left (d(h,f)\right )|h\inv(0) \pitchfork \Sigma^1$ at $p$.

\end{proof}

\begin{prop} \label{cross-cap2}
Take $p\in h\inv(0)$ such that $\rank d(f|h\inv(0))(p)=m-1$ (note that this occurs if and only if $\rank d(h,f)(p)=m-1+l$). The point $p$ is a cross--cap of $f|h\inv(0)$ if and only if $p$ is a regular zero of the mapping $\mu |h\inv(0)$, where $\mu\colon \R^{m+l}\to \R^s$ is given at the point $(x,y)$ by all the  $(m+l)\times (m+l)$--minors of $d(h,f)(x,y)$.
\end{prop}
\begin{proof}
By Proposition \ref{cross-cap1} we get that $p$ is a cross--cap if and only if $d(h,f)|h\inv(0) \pitchfork \Sigma^1 \subset M_{m+l}(\R^{2m-1+l})$ at $p$.

Take $\Phi\colon (M_{m+l}(\R^{2m-1+l}),d(h,f)(p))\to \R^s$ such that $\Phi(A)$ is given by all $(m+l)\times (m+l)$--minors of $A$. Note that $\Phi\inv(0)=(\Sigma^1,d(h,f)(p))$. Similarly as in \cite[proof of Lemma 2]{krzyz} one can show that $\rank d\Phi (d(h,f)(p))=m$. According to \cite[Lemma 1]{krzyz} we obtain, that $d(h,f)|h\inv(0) \pitchfork \Sigma^1$ at $p$ if and only if $\rank d(\Phi\circ d(h,f)|h\inv(0))(p)=m$. Note that  $\Phi\circ d(h,f)|h\inv(0)=\mu|h\inv(0)$, so we get that $p$ is a cross--cap of $f|h\inv(0)$ if and only if $p$ is a regular zero of $\mu |h\inv(0)$.
\end{proof}

Let us take a smooth map $g\colon \R^{m+l}\to \R$ such that $M=h\inv(0)\cap \{ g\geqslant 0\}$ is an $m$--dimensional compact manifold with boundary. We fix an orientation of $M$ as follows. Let $x\in M$. We say, that vectors $v_1,\ldots v_m\in \R^{m+l}$ are well--oriented in $T_xM$ if and only if $\nabla h_1(x), \ldots ,\nabla h_j(x), v_1,\ldots,v_m$ are positively oriented in $\R^{m+l}$. From Proposition \ref{cross-cap1}  we obtain the following.

\begin{theorem} \label{boundary}
Let $f\colon \R^{m+l}\to \R^{2m-1}$ be a smooth map such that $f|M$ has no singular points near $\partial M$.
If $m$ is \textbf{even} then the number of cross--caps in $M$ of every
smooth mapping $\hat{f}\colon M\longrightarrow \R^{2m-1}$ close enough to $f|M$ with only
cross--caps as singular points is congruent to  $\operatorname{I}_2(d(h,f)|M,\Sigma) \mod 2$.
\end{theorem}

\begin{prop} \label{sum}
Let $f\colon \R^{m+l}\to \R^{2m-1}$ be a smooth map such that $f|M$ has only cross--caps as singular points and finitely many of them, moreover no singular points belongs to $\partial M$.
If $m$ is \textbf{odd} then the algebraic sum of cross--caps of $f|M$ is equal to  $\displaystyle \frac{(-1)^{l+1}}{2}\sum\limits _{(\beta,p)} \deg (\widetilde{d(h,f)}|M,(\beta,p))$, where $(\beta,p)'s$ are the zeros of $\widetilde{d(h,f)}|M$. 
\end{prop}
\begin{proof}
In this proof we will use \cite[Theorem 1]{krzyznowel} and Lemma \ref{tech3}. We will define some mappings to the space of matrices, and we will present connections between local invariants associated with them.

Let us take $p\in M$ --- a cross-cap of $f|M$.

Take an orientation--preserving diffeomorphism $\varphi\colon (\R^{m+l},0)\to (\R^{m+l},p)$ such that $\varphi (\R^m\times \{0 \},0)=(M,p)$. Put $\varphi|\R^m(x)=\varphi(x,0)$, $x\in (\R^m,0)$. 
Note that $h\circ \varphi |\R^m\equiv 0$.

One can choose $\varphi$ (composing, if necessary, with an orientation--preserving linear diffeomorphism in the domain) such that the first $m$ columns of the matrix $d\varphi(0)$ are orthogonal. Then it is easy to check that the local coordinate system $\varphi|\R^m$ of $M$ is orientation--preserving if and only if $\sgn \det \left [ dh(p) \frac{\partial \varphi}{\partial y}(0)\right ]=(-1)^l$.

According to \cite[Theorem 1]{krzyz} the sign of $p$ is equal to 
\[
-\frac{1}{2}(\deg (\widetilde{d(f\circ \varphi|\R^m)},(\gamma,0))+\deg (\widetilde{d(f\circ \varphi|\R^m)},(-\gamma,0))),
\]
where $(\gamma, 0)$, $(-\gamma, 0)$ are the only zeros of $\widetilde{d(f\circ \varphi|\R^m)}$. 

Let us define $a,c\colon (\R^m,0)\to M_{m+l}(\R^{2m-1+l})$ as $c(x)=d(h,f)(\varphi(x,0))$ and $a(x)=c(x)\cdot d\varphi(x,0)$. Then both $\widetilde{a}$ and $\widetilde{c}$ have only two zeros, $\widetilde{a}\inv (0)=\{((\gamma,0),0),((-\gamma,0),0)\}$, and from Lemma \ref{tech3} all the local topological degrees of them at these zeros are equal.

Using the definition of the local topological degree, after some computations one can show that 
\[
-\frac{1}{2}\sgn \det \left [ dh(p) \frac{\partial \varphi}{\partial y}(0)\right ]\left(\deg (\widetilde{d(f\circ \varphi|\R^m)},(\gamma,0))+\deg (\widetilde{d(f\circ \varphi|\R^m)},(-\gamma,0))\right)=
\]
\[
=-\frac{1}{2}\left(\deg (\widetilde{a},((\gamma,0),0))+\deg (\widetilde{a},((-\gamma,0),0))\right).
\]
We get that the sign of the cross--cap $p$ equals
\[
-\frac{1}{2}\sgn \det \left [ dh(p) \frac{\partial \varphi}{\partial y}(0)\right ] \sum\limits _{(\delta,0)\in \widetilde{c}\inv (0)} \deg (\widetilde{c},(\delta,0)). 
\]
Note that $\widetilde{c}$ is $\widetilde{d(h,f)}|M$ composed with the local coordinate system $\varphi|\R^m$, so local topological degrees of $\widetilde{c}$ and $\widetilde{d(h,f)}|M$ at its zeros $(\delta,0)$ and $(\beta,0)$ are equal if and only if $\varphi|\R^m$ is orientation--preserving. Then we obtain the conclusion of this Proposition.
\end{proof}

From Propositions \ref{cross-cap1} and \ref{sum} and Theorem \ref{index_even} we obtain the following.

\begin{theorem} \label{boundary2}
Let $m$ be \textbf{odd}. There is such an orientation of $\Sigma^1\subset M_{m+l}(\R^{2m-1+l})$ that the following is true.

Let $f\colon \R^{m+l}\to \R^{2m-1}$ be a smooth map such that $f|M$ has no singular points near $\partial M$.
The algebraic sum of cross--caps in $M$ of every
smooth mapping $\hat{f}\colon M\longrightarrow \R^{2m-1}$ close enough to $f|M$ with only
cross--caps as singular points is equal to  $\operatorname{I}(d(h,f)|M,\Sigma)$.
\end{theorem}

Using \textsc{Singular} (\cite{singular}), Theorems \ref{formula}, \ref{boundary}, \ref{boundary2}, and Proposition \ref{cross-cap2} we present the following examples.

\begin{ex}
Let $M$ be a half of the surface of genus $2$ (double torus), where $g(x,y,z)=z$, $h(x,y,z)=(x(x-1)^2(x-2)+y^2)^2+z^2-0,01\colon \R^3\to \R$. Take $f(x,y,z)=(6yz+2x, 6yz+4y, 8z^2+5z)\colon \R^3\to \R^3$.
Then $f|M$ has only cross--caps as singular points and its number equals $0$ mod 2.
\end{ex}

\begin{ex}
Let $M$ be as in the previous example. Take $f(x,y,z)=(x^2+8x, 10xy+10x, 9z^2+9z)\colon \R^3\to \R^3$. The map $f|M$ has not only cross--caps as singularities, although every map close enough to $f|M$ with only cross--caps as singularities has the number of cross--caps congruent to $1$ modulo $2$.
\end{ex}

Since the number of cross--caps modulo $2$ of the mappings from the two previous examples are not the same, according to \cite{whitney2} these mappings are not homotopic by a homotopy regular near the boundary.

\begin{ex}
Let $h(x,y,z,w)=x^2+y^2+z^2-w^2-1$, and $g_r(x,y,z,w)=r^2-x^2-y^2-z^2-w^2$. Let us define $M_r=h\inv (0)\cap\{g_r\geqslant 0\}$. Then for $r>1$, $M_r$ is a $3$--dimensional manifold with boundary, and the boundary has $2$ connected components. Put $f\colon \R^4\to \R^5$ as $f(x,y,z,w)=(8z^2-4x+4y,4xy+4z,xz-7y+3w,z^2-3x+8y,zw^2+y^2)$. Then mappings $f|M_2$, $f|M_3$, $f|M_{10}$ have only cross--caps as singular points, and no singular points on the boundary. Moreover the algebraic sum of cross--caps of $f|M_2$ equals $1$, of $f|M_3$ equals $0$, and of $f|M_{10}$ equals $-3$.
\end{ex}

So there is no homotopy regular near the boundary between any two of these mappings.

{\bf Acknowledgement.}
The authors wish to express their gratitude to Zbigniew Szafraniec for comments which
have improved the paper.

\end{document}